\documentclass[11pt,a4paper]{article}
\usepackage{epsf,epsfig,amsfonts,amsgen,amsmath,amstext,amsbsy,amsopn,amsthm,lineno}
\usepackage{color}
\usepackage{graphicx}
\usepackage{subfigure}
\usepackage{epstopdf}
\usepackage{authblk}
\usepackage{amsmath,amssymb,amsmath,amsfonts,amscd}
\usepackage{indentfirst,graphicx,epstopdf}
\usepackage{psfrag,ifpdf,enumerate}
\usepackage{caption}
\usepackage{color}
\usepackage[bookmarksnumbered, plainpages]{hyperref}
\usepackage[numbers,sort&compress]{natbib}
\usepackage[noend]{algpseudocode}
\usepackage{algorithmicx,algorithm}
\usepackage{psfrag}
\usepackage{color}
\usepackage{enumerate}
\usepackage{epstopdf}

\newtheorem{theorem}{Theorem}

\newtheorem{lemma}{Lemma}

\newtheorem{construction}{Construction}

\theoremstyle{definition}
\newtheorem{definition}{Definition}
\newtheorem{claim}{Claim}

\newtheorem{conjecture}{Conjecture}
\newtheorem{problem}{Problem}

\begin{document}
	\title
	{\bf\Large A revisit to Bang-Jensen-Gutin conjecture and Yeo's theorem}
	
\date{}
\author[1,2]{\small Ruonan Li\thanks{Supported by NSFC (No. 11901459,
No. 12071370, No. 12131013). E-mail: rnli@nwpu.edu.cn}}
\author[3]{\small Bo Ning\thanks{Supported by NSFC (No. 11971346). E-mail:~bo.ning@nankai.edu.cn}}
\affil[1]{ School of Mathematics and Statistics,
Northwestern Polytechnical University,
Xi'an, 710129, P.R.~China}
\affil[2]{Xi'an-Budapest Joint Research Center for Combinatorics\\
Northwestern Polytechnical University, Xi'an, 710129, P.R.~China}
\affil[3]{School of Cyber Science, Nankai University, Tianjin, 
	300350, P.R.~China}
\maketitle
\vskip -10pt
\begin{abstract}
A path (cycle) is properly-colored if consecutive edges are of distinct colors. 
In 1997, Bang-Jensen and Gutin conjectured a necessary and sufficient condition 
for the existence of a Hamilton path in an edge-colored complete graph. 
This conjecture, confirmed by Feng, Giesen, Guo, Gutin, Jensen
and Rafley in 2006, was laterly playing an important
role in Lo's asymptotical proof of Bollob\'as-Erd\H{o}s' conjecture
on properly-colored Hamilton cycles. In 1997, Yeo obtained a structural 
characterization of edge-colored graphs that containing no properly 
colored cycles. This result is a fundamental tool in the study of 
edge-colored graphs. In this paper, we first give a much shorter proof of the
Bang-Jensen-Gutin Conjecture by two novel absorbing lemmas. We also prove a new sufficient 
condition for the existence of a properly-colored cycle and 
then deduce Yeo's theorem from this result and a closure concept in edge-colored graphs.
\end{abstract}
\noindent
\section{Introduction}
A graph $G$ is called an {\it edge-colored graph} if each edge $e$ is assigned
a color $col_G(e)$ (not necessarily distinct). Denote by $col(G)$ the 
set of colors assigned on $E(G)$. For a vertex $v$ and a vertex set $S$, 
we use $col_G(v,S)$ to denote the colors assigned to
edges between $v$ and $S-v$ in $G$. In particular, define 
$N_G^c(v)=col_G(v,V(G)\setminus\{v\})$.  The subscript $G$ is sometimes
omitted when there is no ambiguity. An edge-colored graph $G$ is called
{\it $k$-colored} if $|col(G)|=k$ and {\it properly colored} (or {\it PC}) if each pair of
adjacent edges are of distinct colors. For terminology and notations 
not defined here, we refer the reader to \cite{BM:08}.

Hamilton problem is one of the most attractive problem in graph theory,
not only on structural analysis  but also on algorithmic complexity.
For undirected graphs, Hamiltonian property and related properties,
such as traceability and pancyclicity, are widely studied. As a result,
many classical methods are found or well used. For instance,
Pos\'{a}'s rotation-extension method, closure method and Menger's theorem.

For Hamilton problem in digraphs, even refining to multipartite tournaments,
there are many unsolved questions (see the survey paper of Volkmann \cite{V07}).
One could naturally ask what can be said about the Hamiltonian property of a highly
connected multipartite tournament containing a cycle factor. Bang-Jensen, Gutin
and Huang \cite{BGH:96} showed that there is no constant $k$ such that every
$k$-connected multipartite tournament with a cycle factor contains a Hamilton
cycle. However, for Hamilton path, Gutin \cite{Gutin:93} proved that every
multipartite tournament with a {\it directed 1-path-cycle factor}, i.e. a
spanning subgraph which is a union of a directed path and a number of directed cycles,
contains a directed Hamilton path.

Given a multipartite tournament $T_p$ with partite sets $V_1,V_2,\ldots,V_p$,
define an edge-colored complete graph $K$ with $V(K)=V(T_p)$ and for each edge $e=uv$,
if $u,v\in V_i$ or $u\in V_i$, $v\in V_j$ and $uv\in A(T_p)$, then assign color $i$
to $e$. Li et al. \cite{LBZ:20} showed that each PC cycle in $K$ is a directed cycle
in $T_p$ and vice versa. In this sense, the class of edge-colored complete graphs
generalizes the class of multipartite tournaments. Therefore, the relation between
PC cycle factors and PC Hamilton properties was studied. Bang-Jensen, Gutin and
Yeo \cite{BGY:98} proved that a PC cycle factor in an edge-colored $K_n$ guarantees
a PC Hamilton path. Bang-Jensen and Gutin \cite{BG:97} further conjectured that
a {\it PC 1-path-cycle factor}, 
i.e. a PC spanning subgraph which is a union
of a PC path and a number of PC cycles, in an edge-colored $K_n$ can also
guarantee a PC Hamilton path\footnote{A PC Hamilton path  can be regarded as a 
	PC 1-path-cycle factor with the number of PC cycles being $0$.}. This conjecture was confirmed by Feng et al.
\cite{FGGGJR:06} in 2006.

\begin{theorem}[Feng, Giesen, Guo, Gutin, Jensen and 
Rafiey \cite{FGGGJR:06}]\label{Thm:Fengetal}
Let $\{P,C_1,C_2,\ldots,C_k\}$ be a PC 1-path-cycle factor of an
edge-colored complete graph $G$ such that $P$ is a PC path and $C_i$ is a PC
cycle for $1\leq i\leq k$. Then $G$ contains a PC Hamilton path.
\end{theorem}

One important open problem on PC cycles is the Bollob\'as-Erd\H{o}s Conjecture.
Let $K^c_n$ be an edge-colored $K_n$ and $\Delta_{mon}(K^c_n)$ be
the maximum number of edges with the same color which are incident to the same
vertex. Confirming a conjecture proposed by Daykin \cite{D:78}, Bollob\'as
and Erd\H{o}s \cite{BE:76} showed that if
$\Delta_{mon}(K^c_n)\leq \frac{n}{69}$ then $K^c_n$
contains a PC Hamilton cycle. In the same paper, Bollob\'as
and Erd\H{o}s \cite{BE:76} conjectured that the weaker condition
that $\Delta_{mon}(K^c_n)<\lfloor\frac{n}{2}\rfloor$ suffices for the existence
of PC Hamilton cycle. In 2016, Lo \cite{L:16} confirmed the conjecture
asymptotically by proving that for any $\varepsilon>0$, every $K^c_n$
with $\Delta_{mon}(K^c_n)\leq (\frac{1}{2}-\varepsilon)n$ contains a
PC Hamilton cycle, provided $n$ is sufficiently large. One main tool
of Lo's proof is Theorem \ref{Thm:Fengetal}. For problems and results
related to Theorem \ref{Thm:Fengetal}, see \cite{BFMMMPS:19,CKW:20}.

In this article, we shall present a much shorter proof of Bang-Jensen-Gutin
conjecture by proving the following theorem, which is Claim A in Feng et 
al.'s proof \cite{FGGGJR:06}. Note that to extend a PC path, the 
end-vertices, starting and ending colors are important. Hence we 
define the {\it character}
of a PC path  $P=u_1u_2\cdots u_m$ as $ch(P)=[u_1,col(u_1u_2),col(u_{m-1}u_m),u_m].$ We regard $[u_1,col(u_1u_2),col(u_{m-1}u_m),u_m]$ and $[u_m,col(u_{m-1}u_m),col(u_1u_2),u_1]$
as the same object. So each PC path has a unique character.

\begin{theorem}[Claim A in \cite{FGGGJR:06}]\label{Thm:key-1}
Let $\{P,C\}$ be a 1-path-cycle factor of an edge-colored $K^c_n$ 
such that $P$ is a PC path and $C$ is a PC
cycle. If $ch(P)=[u_1,x,y,u_m]$ ($m\geq 2$),
$col(u_1,V(C))=\{x\}$ and $col(u_m,V(C))=\{y\}$. Then $G$ 
contains a PC Hamilton path $Q$ with $ch(Q)=ch(P)$.
\end{theorem}

Another natural problem is to find the sufficient 
sharp color degree condition for the existence
of a PC cycle in an edge-colored graph.
In 1997, Yeo \cite{Y:97} proved a local characterization,
which extended a previous theorem of Grossman
and H\"{a}ggkvist \cite{GH:83} to edge-colored graphs
of arbitrary number of colors by a quite technical method.
\begin{theorem}[Yeo \cite{Y:97}]\label{Y:97}
Let $G$ be an edge-colored graph containing no PC cycle. Then
there exists a vertex $z\in V(G)$ such that no component of $G-z$ is
joint to $z$ with edges of more than one color.
\end{theorem}
With aid of Yeo's theorem, Fujita, the first author, and Zhang \cite{FLZ:18}
proved that for an edge-colored graph $G$ with 
$\min_{v\in V(G)}|N^c_G(v)|\geq D$, if $D!\sum_{i=0}^D\frac{1}{i!}>n$ then $G$ contains a PC cycle.
Furthermore, the bound is best possible.
Besides, Yeo's theorem is frequently used not only on the existence of PC
cycles (see the survey paper \cite{BG:97} and recent
results \cite{COY:18,FLZ:18,GJSWY:17,LBYY:21}) but also on other
topics, such as PC trees \cite{BFMMMPS:19},
decomposition of edge-colored graphs \cite{FLW:19},
Ramsey numbers \cite{ORR:19} and matchings \cite{WYZ:16}.

We present a new sufficient condition
for the existence of a PC cycle in an edge-colored graph.

\begin{theorem}\label{Thm:key-2}
Let $G$ be an edge-colored graph. If for each vertex $v$, there exists a
vertex $v'$ and two internally vertex-disjoint PC paths $P$ and $Q$ 
such that $ch(P)=[v,\alpha_1,\beta_1,v']$, $ch(Q)=[v,\alpha_2,\beta_2,v']$ 
and $\alpha_1\neq \alpha_2$, then $G$ contains a PC cycle.
\end{theorem}

By introducing a new
concept of closure for properly-colored cycles in edge-colored
graphs, we will give a more intuitive and slightly shorter proof of
Theorem \ref{Y:97} from Theorem \ref{Thm:key-2}.

\section{Proofs of Theorems \ref{Thm:Fengetal} and \ref{Thm:key-1}}
In this section, $P$ and $C$ always represent a path and a cycle,
respectively. For an edge $vv'\in E(C)$, we use $vCv'$ to denote the
unique path on $C-vv'$ from $v$ to $v'$. For two vertices
$v,v'\in V(P)$, let $vPv'$ be the unique path on $P$ from $v$ to $v'$.
For a vertex $v\not\in V(P)$, if $col(vu_{r-1})=col(u_{r-1}u_r)$
and $col(vu_{r+1})=col(u_{r}u_{r+1})$ for some $r\in [2,m-1]$,
then we say $v$ {\it replaces} $u_r$ on $P$. Let $t(v)$ be
the smallest integer such that $col(vu_{t(v)})\neq col(u_{t(v)}u_{t(v)+1})$
and let $s(v)$ be the largest integer such that
$col(vu_{s(v)})\neq col(u_{s(v)}u_{s(v)-1})$. For a path
$W=w_1w_2\cdots w_k$ $(k\geq 1)$ with $V(W)\cap V(P)=\emptyset$,
if $u_1Pu_iw_1Ww_ku_{i+1}Pu_{m}$ is a PC path for some
$i\in [1,m-1]$, then we say $P$ {\it absorbs $W$ in the
form }$[u_i,w_1,w_k,u_{i+1}]$.	

Our short proof of Theorem \ref{Thm:Fengetal} mainly relies on the following two new technical lemmas.

\begin{lemma}\label{lem:AB}
Let $\{P,C\}$ be a  PC 1-path-cycle factor of an edge-colored 
graph $G$ such that $P=u_1u_2u_3$. If there exists
$vv'\in E(C)$ such that $v$ can replace $u_2$ on
$P$ and $u_2v'\in E(G)$, then $P$ absorbs $vCv'$
or $vv'$ in the form $[u_1,v,v',u_2]$ or
$[u_2,v',v,u_3]$. Denote by $Q$ the obtained PC path.
Then $ch(Q)=ch(P)$.		
\end{lemma}
\begin{proof}
Let $[v,a,b,v']$ be the character of the PC path $vCv'$. Let $col(vv')=c$,
$col(vu_1)=col(u_1u_{2})=RED$ and $col(vu_{3})=col(u_{2}u_{3})=BLUE$. Consider the
color of $u_{2}v'$.
		
If $col(u_{2}v')\not\in \{RED,BLUE\}$, then by the symmetry between RED
and BLUE, assume that $a\neq RED$. Define (see Figure \ref{fig:Lemma1}(a)-(c))
       \[Q=
		\begin{cases}
			u_1vCv'u_{2}u_3,&\text{if  $col(u_{2}v')\neq b$;}\cr
			u_1vv'u_{2}u_3,&\text{if $col(u_{2}v')=b$ and $c\neq RED$;}\cr
			u_1u_{2}v'vu_3,&\text{if $col(u_{2}v')=b$ and $c\neq BLUE$.}\cr
		\end{cases}
		\]	
		
If $col(u_{2}v')\in \{RED,BLUE\}$, then again by the symmetry 
between RED and BLUE, assume that $col(u_{2}v')=RED$. Define	
(see Figure \ref{fig:Lemma1}(d) and (e))
	\[Q=
	\begin{cases}
		u_1vCv'u_{2}u_3,&\text{if $c=RED$;}\cr
		u_1vv'u_{2}u_3,&\text{if $c\neq RED$.}\cr
	\end{cases}
	\]
	Then $Q$ is always a desired PC path.
   \begin{figure}[t]
		\centering
		\subfigure[$col(u_{2}v')\neq b$]{
			\includegraphics[width=0.145\textwidth]{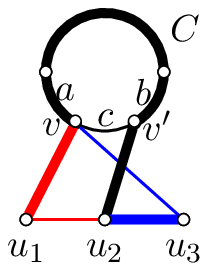}}
		\hskip 20pt
		\subfigure[$col(u_{2}v')=b$ and $c\neq RED$]{
			\includegraphics[width=0.145\textwidth]{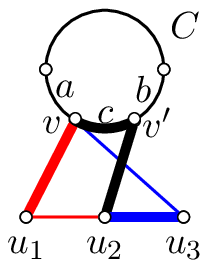}}
		\hskip 20pt
		\subfigure[$col(u_{2}v')=b$ and $c\neq BLUE$]{
		\includegraphics[width=0.145\textwidth]{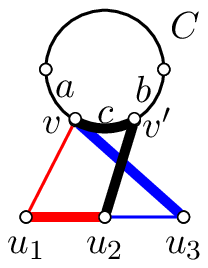}}
		\hskip 20pt
		\subfigure[$col(u_{2}v')=RED$ and $c=RED$]{
		\includegraphics[width=0.145\textwidth]{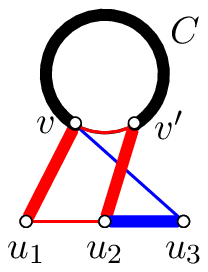}}
		\hskip 20pt
		\subfigure[$col(u_{2}v')=RED$ and $c\neq RED$]{
			\includegraphics[width=0.145\textwidth]{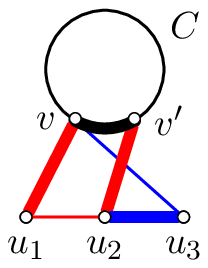}}
		\caption{Three cases in Lemma 1 \label{fig:Lemma1}}
	\end{figure}
	\end{proof}
With the aid of Lemma \ref{lem:AB} and the definitions of $t(v)$ and $s(v)$,
we obtain the following lemma which can be viewed as a weaker version of
Theorem \ref{Thm:key-1}.
\begin{lemma}\label{lem:edge}
Let $\{P,C\}$ be a PC 1-path-cycle factor of an edge-colored 
complete graph $G$ with  $P=u_1u_2\cdots u_m$ $(m\geq 2)$ and  
$ch(P)=[u_1,x,y,u_m]$. If for an edge $vv'\in E(C)$, there 
hold $col(u_1,\{v,v'\})=\{x\}$ and $col(u_m,\{v,v'\})=\{y\}$, 
then $G$ contains a PC path $Q$ with the following holds: 
(1)~$ch(Q)=ch(P)$;  (2)~$ V(Q)=V(C)\cup V(P)$ or $V(Q)=\{v,v'\}\cup V(P)$. 
\end{lemma}
\begin{proof}
Let $[v,a,b,v']$ be the character of the PC path $vCv'$.
In the case $m=2$, we have $x=y$. One can easily verify that no matter $col(vv')$
equals $x$ or not, a desired PC path $Q$ always exists. Now assume $m\geq 3$.
If $col(vu_2)=x$, then since $u_1v'Cvu_2Pu_m$ is not a PC path, we have $x=a$ or
$x=b$, which implies that $x\neq col(vv')$. Thus $Q=u_1v{'}vu_2Pu_m$ is a desired
PC path. So for each $w\in \{v,v'\}$, assume $col(wu_2)\neq x$ and $col(wu_{m-1})\neq y$.
Hence $t(w)$ and $s(w)$ are well defined with $t(w),s(w)\in [2,m-1]$.
		
If  $col(wu_{t(w)})=col(wu_{t(w)-1})$ for some $w\in \{v,v'\}$, then  $wu_{t(w)}u_{t(w)-1}w$
is a monochromatic triangle. Note that $col(wu_2)\neq x$. Therefore $t(w)>2$ and $w$
can replace $u_{t(w)-1}$. According to Lemma \ref{lem:AB}, a desired PC path $Q$
exists. Consider the symmetry between $t(w)$ and $s(w)$. The only remaining case is that
$col(wu_{t(v)})\neq col(wu_{t(v)-1})$ and $col(wu_{s(v)})\neq col(wu_{s(v)+1})$
for each $w\in \{v,v'\}$. In this case, we can check that $s(w)\geq t(w)$. If $t(v)\neq t(v')$
or $s(v)\neq s(v')$, then a desired PC path can be found by absorbing
$v$ and $v'$ from distinct segments of $P$. Now assume that $t=t(v)=t(v')$ and $s=s(v)=s(v')$. If $t\neq s$,
then $t<s$ and $P$ also absorbs $\{v,v'\}$. So the only case left is
that $t=s$, which implies that $v$ can replace $u_t$. Applying Lemma \ref{lem:AB} 
to $v$ and the path $u_{t-1}u_tu_{t+1}$,
we find a desired PC path $Q$. 	
\end{proof}
The original proof of Theorem \ref{Thm:key-1} in \cite{FGGGJR:06} is by induction on
$2|C|+|P|$ and reducing the cases into $|C|\in \{3,5\}$. Our proof adheres the proof
structure and many notations, reducing the cases into $|C|=3$. New lemmas dramatically
shorten the induction process. Now we give the new proof of Theorem \ref{Thm:key-1}
as following.
\begin{proof}[{\bf Proof of Theorem \ref{Thm:key-1}}]
Let $P=u_1u_2\cdots u_m$ and $C=v_1v_2\cdots v_nv_1$. Then $col(u_1u_2)=x$ 
and $col(u_{m-1}u_m)=y$. Suppose that $(P,C)$ is a counterexample 
such that $2|C|+|P|$ is as small as possible.

\begin{claim}\label{cl:n-2}
For each edge $vv'\in E(C)$, the following statements hold: \\
(a) there is a PC path $Q$ with $V(P)=V(Q)\cup \{v,v'\}$ and $ch(Q)=ch(P)$;\\
(b) $G[V(C)\setminus \{v,v'\}]$ contains no spanning PC cycle.
\end{claim}
\begin{proof}
Lemma  \ref{lem:edge} implies the existence of $Q$ in ($i$).
Suppose to the contrary that $H$ is a spanning PC cycle of $G[V(C)\setminus \{v,v'\}]$.
Since $2|H|+|Q|<2|C|+|P|$, there must be a PC path $Q'$ on $V(H)\cup V(Q)=V(C)\cup V(P)$
with $ch(Q')=ch(Q)=ch(P)$, a contradiction.
\end{proof}
\begin{claim}\label{cl:2a}
If $col(v_iv_{i+1})=col(v_jv_{j+1})=a$, then $col(v_iv_{j+1})\neq a$ 
and $col(v_jv_{i+1})\neq a$.
\end{claim}
\begin{proof}
If $v_iv_{j+1}\in E(C)$ or $v_jv_{i+1}\in E(C)$, then Claim \ref{cl:n-2}
and the fact that $C$ is a PC cycle imply the statement. Now assume
that $v_iv_{j+1}\not\in E(C)$ and $v_jv_{i+1}\not\in E(C)$. By symmetry,
it suffices to prove that $col(v_jv_{i+1})\neq a$. Suppose not.
Then $C'=v_{i+1}v_{i+2}\cdots v_{j}v_{i+1}$ is a PC cycle. Let $H$ be
the cycle $v_{j+1}v_{j+2}\cdots v_{i}v_{j+1}$ with colors of edges
first adhered from that in $G$ and then recoloring the edge
$v_{i}v_{j+1}$ by color $a$. Obviously, $H$ is a PC cycle and
$P$ can absorb $V(H)$ with $Q$ as a resulting PC path satisfying that
$ch(Q)=ch(P)$. If this new colored edge $v_{i}v_{j+1}$ is contained in
$Q$, then replace the segment $v_{i}v_{j+1}$ in $Q$ by
$v_iv_{i+1}\cdots v_{j}v_{j+1}$,
we obtain a desired PC path, a contradiction. Hence $Q$ must be an
edge-colored subgraph of $G$. Consider the pair $(C',Q)$. Since
$2|C'|+|Q|<2|C|+P$, there must be a PC path $Q'$ on $V(Q)\cup V(C')=V(G)$
with $ch(Q')=ch(Q)=ch(P)$, a contradiction.
\end{proof}
From Claims \ref{cl:n-2} and \ref{cl:2a}, we immediately have the following claim.

\begin{claim}\label{cl:3x}
If $n\geq 6$, then $col(v_{i}v_{i+1})\neq col(v_{i+4}v_{i+5})$ for each $i\in [1,n]$.
\end{claim}
For $i\in [2,m-1]$, let $F_+(u_i)=\{v\in V(C): col(vu_i)\neq col(u_iu_{i+1})\}$
and
$F_-(u_i)=\{v\in V(C): col(vu_i)\neq col(u_iu_{i-1})\}.$
Note that $F_+(u_i)\neq \emptyset$ for each $i\in [2,m-1]$. Otherwise, by
considering the pair $(u_iPu_m,C)$, we obtain a PC path $Q$ from $u_i$ to
$u_m$ with $ch(Q)=ch(u_iPu_m)$, which can be extended into a desired PC
Hamilton path of $G$ by adding the segment $u_1Pu_i$. Similarly,
$F_-(u_i)\neq \emptyset$ for each $i\in [2,m-1]$. In particular,
define $F_-(u_1)=F_+(u_m)=V(C)$. The remaining proof is proceeded by
studying the property of $F_+(u_2),F_+(u_3),F_-(u_{m-1})$ and $F_-(u_{m-2})$.
\begin{claim}\label{cl:tauu2}
If $n\geq 4$, then the following holds.\\
(a) $col(vu_2)\in N_C^c(v)$ and $col(v'u_{m-1})\in N_C^c(v')$ for all $v\in F_+(u_2)$ and $v'\in F_-(u_{m-1})$;\\
(b) $x\not\in col(u_2,V(C))$ and  $y\not\in col(u_{m-1},V(C))$.
\end{claim}
\begin{proof}
(a) Assume $v_1\in F_+(u_2)$. By the symmetry between $u_2$ and $u_{m-1}$,
it suffices to prove $col(v_1u_2)\in \{col(v_1v_2),col(v_nv_1)\}$. Suppose not.
Since $u_1v_2Cv_1u_2Pu_m$ and $u_1v_nCv_1u_2Pu_m$ are not PC paths,
we get $col(v_2v_3)=col(v_{n-1}v_n)=x$, which is a contradiction when $n=4$.
Thus $n\geq 5$. Claim \ref{cl:2a} tells that $col(v_nv_2)\neq x$. 	
Let $col(v_nv_1)=\alpha$ and $col(v_1v_2)=\beta$. Since $\alpha\neq \beta$,
we assume without loss of generality that $col(v_nv_2)\neq \alpha$.
Consider the path $u_1v_{n-1}v_{n-2}\cdots v_2v_nv_1u_2Pu_m$, we get
$col(v_{n-2}v_{n-1})=x$, contradicting that $col(v_{n-1}v_n)=x$.
	
(b) Suppose to the contrary that  $col(v_1u_2)=x$. Apparently, $v_1\in F_+(u_2)$.
Without loss of generality, assume that $col(v_1v_2)\neq x$. Since $u_1v_nCv_1u_2Pu_m$
is not a PC path, we have $col(v_{n-1}v_n)=x$, which forces $col(v_{n}v_1)\neq x$. 
Thus $col(v_1u_2)=x\not\in N^c_C(v_1)$, which contradicts Claim \ref{cl:tauu2}$(a)$. 
By the symmetry between $u_2$ and $u_{m-1}$, we also get $y\not\in col(u_{m-1},V(C))$.	
\end{proof}

Since $F_+(u_2)\neq \emptyset$, in the following, when $n\geq 4$, 
we assume without loss of generality that $v_1\in F_+(u_2)$ and 
$col(u_2v_1)=col(v_1v_2)=z$ and $z\neq x$.

\begin{claim}\label{cl:v2}
If $n\geq 4$, then $s(v_2)\in [3,m-2]$, the vertex $v_2$ replaces $u_{s(v_2)+1}$
and $col(v_2v_3)=col(v_4v_5)=x$ (indices are taken modulo $n$).
\end{claim}
\begin{proof}
Since  $u_1v_2v_3\cdots v_nv_1u_2Pu_m$ is not a PC path (see Figure \ref{fig:claim5}(a)), 
we have $col(v_2v_3)=x$.
We further assert that $col(v_2u_3)\neq col(u_2u_3)$. Otherwise,
$u_1u_2v_1v_n\cdots v_2u_3Pu_m$ is a PC path (see Figure \ref{fig:claim5}(b)).
Recall the definition of $s(v_2)$, we get $s(v_2)\geq 3$. If
$col(v_2u_{s(v_2)})\neq col(v_2u_{s(v_2)+1})$, then
$u_1v_3v_4\cdots v_nv_1u_2Pu_{s(v_2)}v_2u_{s(v_2)+1}Pu_m$ is a PC Hamilton path
(see Figure \ref{fig:claim5}(c)),  a contradiction. Hence
$col(v_2u_{s(v_2)})= col(v_2u_{s(v_2)+1})$, i.e. the triangle
$v_2u_{s(v_2)}u_{s(v_2)+1}v_2$ is monochromatic. If $s(v_2)= m-1$,
then this triangle is of color $y$, contradicting Claim \ref{cl:tauu2}(b).
In summary, $s(v_2)\in [3,m-2]$ and $v_2$ replaces $u_{s(v_2)+1}$ on $P$.
Suppose that $col(v_4v_5)\neq x$. Apply Lemma \ref{lem:AB} to the edge $v_2v_3$
on $C$ and the PC path $Q'=u_{s(v_2)}u_{s(v_2)+1}u_{s(v_2)+2}$. We can see
that either as a segment of $P$, $Q'$ absorbs $v_2Cv_3$ with resulting a
desired PC Hamilton path or $Q'$ absorbs $v_2v_3$ with resulting a PC path
$W$ such that $V(W)=\{u_i: s(v_2)\leq i\leq m\}\cup \{v_2,v_3\}$ and
$ch(W)=[u_{s(v_2)},col(u_{s(v_2)}u_{s(v_2)+1}), y, u_m]$. Let $Q=u_1v_4v_5
\cdots v_1u_2u_3\cdots u_{s(v_2)}Wu_m$. Then $Q$ is a PC Hamilton
path with $ch(Q)=ch(P)$, a contradiction.
\end{proof}
\begin{figure}[h]
	\centering
	\subfigure[$col(v_2v_3)\neq x$]{
		\includegraphics[width=0.26\textwidth]{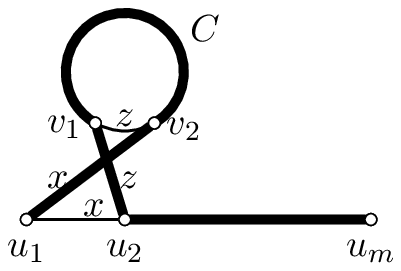}}
	\hskip 15pt
	\subfigure[$col(v_2u_3)=col(u_2u_3)$]{
		\includegraphics[width=0.26\textwidth]{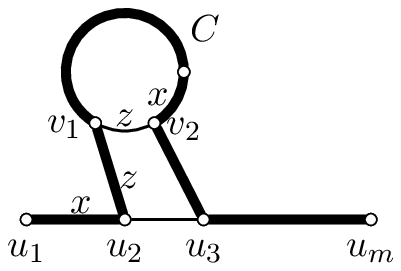}}
	\hskip 15pt
	\subfigure[$col(v_2u_{s(v_2)})\neq col(v_2u_{s(v_2)+1})$]{
		\includegraphics[width=0.32\textwidth]{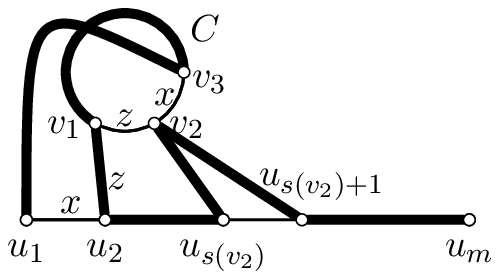}}
	\caption{Cases in the proof of Claim \ref{cl:v2} \label{fig:claim5}}
\end{figure}

\begin{claim}
$n=3$.
\end{claim}
\begin{proof}
If $n\geq 5$, then Claims \ref{cl:2a} and \ref{cl:v2} force $C'=v_2v_3v_4v_5v_2$ being a PC cycle. Claim \ref{cl:3x} tells $col(v_6v_7)\neq x$ when $n\geq 6$. Let $Q=u_1v_6v_7\cdots v_1u_2Pu_m$ (particularly, $Q=u_1v_1u_2Pu_m$ when $n=5$). Thus $(Q,C')$ is a counterexample with $2|C'|+|Q|<2|C|+|P|$, a contradiction.

In the case $n=4$, Claim \ref{cl:v2} tells that $C$ contains two edges of color $x$.
Let $v^*$ be a vertex in $F_+(u_3)$. Rename the vertices of $C$ such that $C=v^*wpqv^*$
and $col(v^*w)=x$. Since $x\not\in col(u_2,V(C))$ (by Claim \ref{cl:tauu2}(b)), if
$col(u_3v^*)\neq x$, then $Q=u_1u_2qpwv^*u_3Pu_m$ is a desired PC path, a
contradiction. Hence $col(u_3v^*)=x$. Define a PC path as follows.
\[
Q=\begin{cases}
	u_1u_2wpqv^*u_3Pu_m, &\text{if $col(u_2w)\neq col(wp)$;}\cr
	u_1wu_2pqv^*u_3Pu_m, &\text{if  $col(u_2w)\neq col(u_2p)$;}\cr
    u_1wpqu_2v^*u_3Pu_m, &\text{if $col(u_2v^*)\neq col(u_2q)$.}\cr
\end{cases}
\]	
If one of the above conditions holds, then $Q$ is a desired PC path.
Therefore we assume that
$u_2wpu_2$ is a monochromatic triangle of color $a$ and $col(u_2v^*)=col(u_2q)=b$.
If $a\neq col(u_2u_3)$, then choose $w$ as $v_1$, and consequently $p$ plays
the role of $v_2$. If $b\neq col(u_2u_3)$, then by Claim \ref{cl:tauu2} we have
$v^*\in F_+(u_2)$ and $col(qv^*)=b$. By a similar argument, the vertex $q$ can
play the role of $v_2$. Since $F_+(u_2)\neq \emptyset$, we have $a\neq col(u_2u_3)$
or $b\neq col(u_2u_3)$. Recall the property of $v_2$ in Claim \ref{cl:v2}, either
$p$ or $q$ can replace $u_{s(v_2)+1}$ on $P$ with $s(v_2)\in [3,m-2]$. Apply
Lemma 1 to the edge $pq$ and the segment $u_{s(v_2)}u_{s(v_2)+1}u_{s(v_2)+2}$,
we see that the path $u_3Pu_m$ absorbs the edge $pq$ with resulting a PC path
$Q'$ such that $ch(Q')=[u_3,col(u_3u_4),y,u_m]$. If $a\neq b$, then
$Q=u_1wu_2v^*u_3Q'u_m$ is a PC Hamilton path with $ch(Q)=ch(P)$, a contradiction.
Hence $a=b$ and all the edges between $u_2$ and $C$ have a same color, i.e.
color $z\neq col(u_2u_3)$. Choose a vertex $v'\in F_-(u_{m-1})$. By the
symmetry between $(u_2,x)$ and $(u_{m-1},y)$, we get $col(v'u_{m-1})\neq y$
and $u_{m-2}u_{m-1}v'u_m$ is a PC path. Rename the vertices on $C$ again
with assuming $C=v'w_1w_2w_3v'$ and $col(v'w_1)=col(w_2w_3)=x$. Then
$Q=u_1w_1w_2w_3u_2Pu_{m-2}u_{m-1}v'u_m$ is a PC Hamilton path with $ch(Q)=ch(P)$, a contradiction.
\end{proof}

Now we deal with the only remaining case $n=3$. Let $col(v_iv_{i+1})=i+2$
(indices are taken modulo $3$) for $i=1,2,3$. Let $col(u_2u_3)=\alpha$.
Without loss of generality, assume that $v_1\in F_+(u_2)$. Since neither
$u_1v_3v_2v_1u_2Pu_m$ nor $u_1v_2v_3v_1u_2Pu_m$ is a PC path, we have $x=1$.
Thus $x\not\in N^c_C(v_1)$.	Note that $x\in N^c_C(v_2)\cap N^c_C(v_3)$.
We have $v_2,v_3\not\in F_+(u_2)$, namely $col(v_2u_2)=col(v_3u_2)=\alpha$.
Apply Lemma \ref{lem:edge} to $v_2v_3$ and the PC path $u_2Pu_m$.
Since there is no PC path on $V(C)\cup V(P)$, we get a PC path $P'$
on $(V(C)\setminus\{v_1\})\cup (V(P)\setminus \{u_1\})$ with
$ch(P')=ch(u_2Pu_m)$. If $col(v_1u_2)\neq x$, then $Q=u_1v_1u_2P'u_m$
is a PC path on $V(C)\cup V(P)$ with $ch(Q)=ch(P)$, a contradiction.
Thus $u_1u_2v_1u_1$ is a monochromatic triangle of color $x=1$.
Now consider colors in $col(u_3,V(C))$. If $v_1\in F_+(u_3)$, then
either $u_1u_2v_3v_2v_1u_3Pu_m$ or $u_1u_2v_2v_3v_1u_3Pu_m$ is a
desired PC path, a contradiction. Thus $v_1\not\in F_+(u_3)$. If
$v_2\in F_+(u_3)$, then discuss the color of $v_2u_3$. 			
Since neither $u_1v_3u_2v_1v_2u_3Pu_m$ nor $u_1v_3v_1u_2v_2u_3Pu_m$
is a PC path, we get $col(v_2,\{v_1,u_2,u_3\})=\{\alpha\}=\{3\}$. 			
If $v_3\in F_+(u_3)$, then similarly, we can get
$col(v_3,\{v_1,u_2,u_3\})=\{\alpha\}=\{2\}$. Therefore, only one
of $v_2$ and $v_3$ is in $F_+(u_3)$. Hence, in summary, without
loss of generality, we assume that $F_+(u_2)=\{v_1\}$,
$col(v_1u_2)=x=1$, $F_+(u_3)=\{v_2\}$ and $col(v_2u_3)=\alpha=3$.

\begin{figure}[h]
	\centering
	\subfigure[$y=1$]{
		\includegraphics[width=0.45\textwidth]{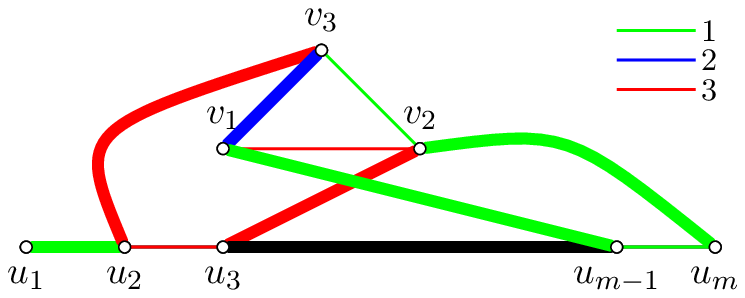}}
	\hskip 30pt
	\subfigure[$y=3$ and $col(u_2u_{m-1})=3$]{
		\includegraphics[width=0.45\textwidth]{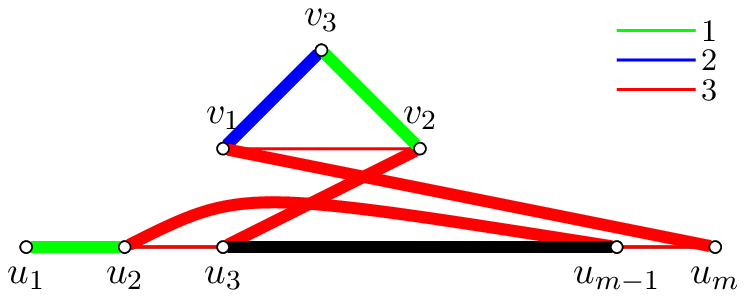}}
	\vskip 10pt
	\subfigure[$y=3$, $col(u_2u_{m-1})\neq 3$ and $\beta=1$]{
		\includegraphics[width=0.45\textwidth]{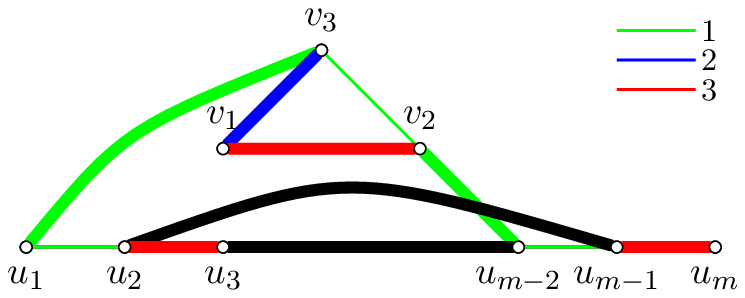}}
	\hskip 30pt
	\subfigure[$y=3$, $col(u_2u_{m-1})\neq 3$ and $\beta=2$]{
		\includegraphics[width=0.45\textwidth]{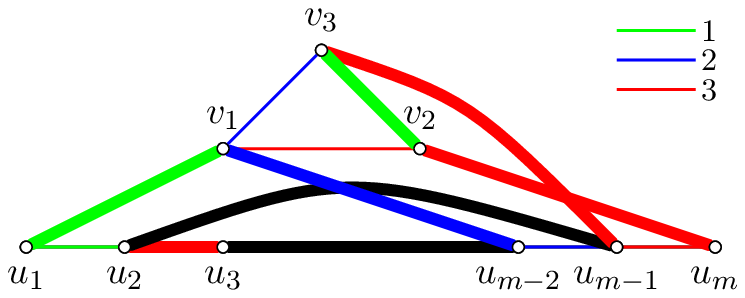}}
	\caption{Four cases when $y\in \{x,\alpha\}$ \label{fig:n3Q1}}
\end{figure}
Let $col(u_{m-2}u_{m-1})=\beta$.
By a similar argument, we can assume $F_-(u_{m-1})=\{v_i\}$ and 
$F_-(u_{m-2})=\{v_j\}$.  Then $\{y,\beta\}\in \{1,2,3\}$, $i\neq j$, $y=i$
and $ \{i,j,\beta\}=\{1,2,3\}$.	

Note that there is exactly one edge in $\{v_iu_3:i=1,2,3\}$ of color $\alpha$,
and the other two edges are of color $col(u_3u_4)$. By the symmetry between $(u_2,x)$
and $(u_{m-1},y)$, there is only one edge in $\{v_iu_{m-1}:i=1,2,3\}$ of color
$y$, and the other two edges are of color $\beta$. This implies that
$u_3\neq u_{m-1}$. Thus $m\geq 5$.

If $y\in \{x,\alpha\}$, then define a PC path as follows (see Figure \ref{fig:n3Q1}).\[
Q=\begin{cases}
	u_1u_2v_3v_1u_{m-1}Pu_3v_2u_m, &\text{if $y=1$;}\cr
	u_1u_2u_{m-1}u_{m-2}Pu_3v_2v_3v_1u_m, &\text{if $y=3$ and $col(u_2u_{m-1})=3$.}\cr
	u_1v_3v_1v_2u_{m-2}Pu_2u_{m-1}u_jm, &\text{if $y=3$, $col(u_2u_{m-1})\neq 3$ and $\beta=1$.}\cr
	u_1v_1u_{m-2}Pu_2u_{m-1}v_3v_2u_m, &\text{if $y=3$, $col(u_2u_{m-1})\neq 3$ and $\beta=2$.}\cr
\end{cases}\]
In all cases above, $Q$ is a PC path on $V(C)\cup V(P)$ with $ch(Q)=ch(P)$, 
a contradiction. Hence, $y\not\in \{x,\alpha\}$ and similarly 
$x\not\in \{y,\beta\}$. So $y=i=2$, $\beta=3$, $j=1$ and $m\geq 6$. 
If $col(u_3u_{m-1})=3$, then $u_1v_2u_{m-1}u_3Pu_{m-2}v_1u_2v_3u_m$ 
is a desired PC path (see Figure \ref{fig:fig1a}), a contradiction.
\begin{figure}
	\centering
	\includegraphics[width=0.45\linewidth]{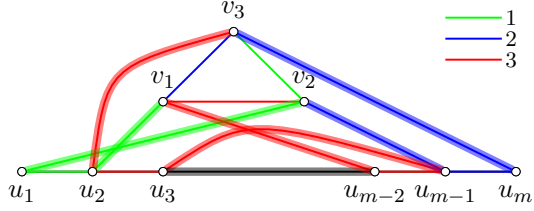}
	\caption{$col(u_3u_{m-1})=3$}
	\label{fig:fig1a}
\end{figure}

Now $col(u_3u_{m-1})\neq 3$. We will complete the proof by distinguishing $F^+(u_4)$. 
Define a PC path as following (see Figure \ref{fig:n3Q2}).
 \begin{figure}[h]
 	\centering
 	\subfigure[$v_1\in F^+(u_4)$ and $col(v_1u_4)=2$]{
 		\includegraphics[width=0.45\textwidth]{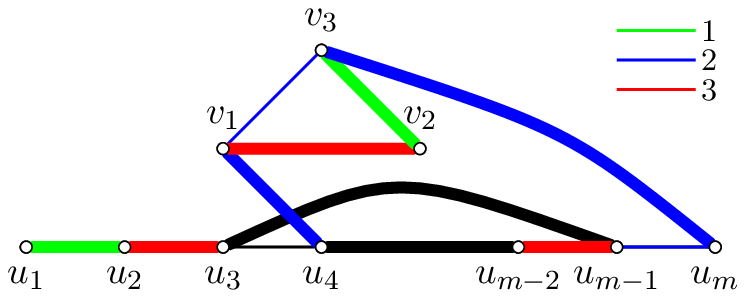}}
 	 	\hskip 30pt
 	\subfigure[$v_1\in F^+(u_4)$ and $col(v_1u_4)\neq 2$]{
 		\includegraphics[width=0.45\textwidth]{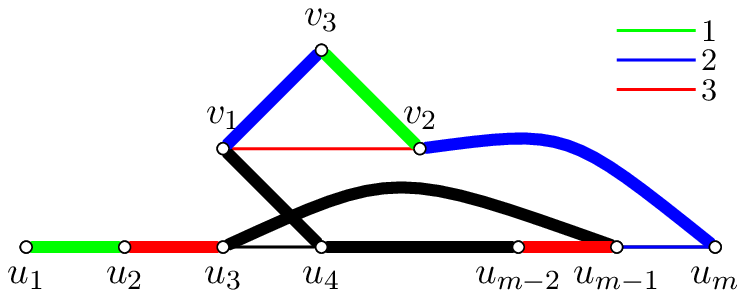}}
 	\vskip 8 pt
 	\subfigure[$v_3\in F^+(u_4)$ and $col(v_3u_4)=2$]{
 		\includegraphics[width=0.45\textwidth]{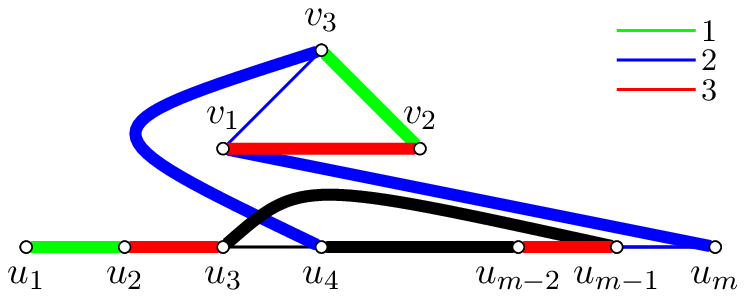}}
 	 	\hskip 30pt
 	\subfigure[$v_3\in F^+(u_4)$ and $col(v_3u_4)\neq 2$]{
 	 	\includegraphics[width=0.45\textwidth]{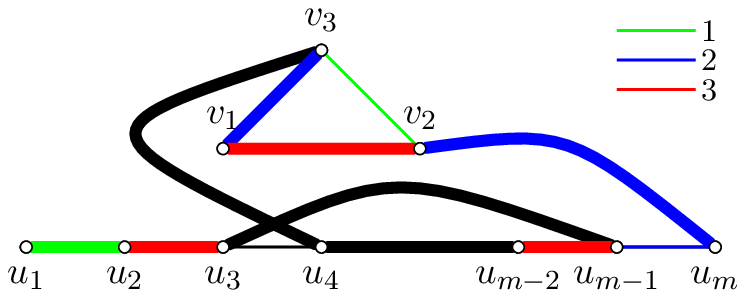}}
  	 	\vskip 8 pt
  	\subfigure[$v_2\in F^+(u_4)$ and $col(v_2u_4)=2$]{
  		\includegraphics[width=0.45\textwidth]{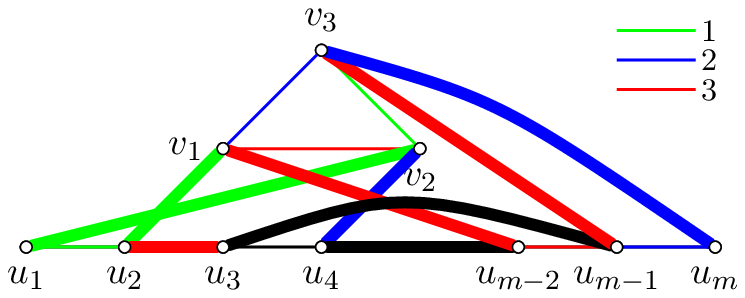}}
  	\hskip 30pt
  	\subfigure[$v_2\in F^+(u_4)$ and $col(v_2u_4)\neq 2$]{
  		\includegraphics[width=0.45\textwidth]{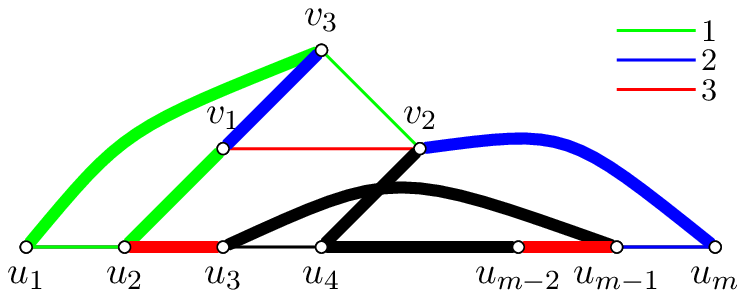}}
 	\caption{Six cases when $col(u_3u_{m-1})\neq 3$ \label{fig:n3Q2}}
 \end{figure}
\[Q=\begin{cases}
	u_1u_2u_3u_{m-1}Pu_4v_1v_2v_3u_m, &\text{if $v_1\in F^+(u_4)$ and $col(v_1u_4)=2$.}\cr
	u_1u_2u_3u_{m-1}Pu_4v_1v_3v_2u_m, &\text{if $v_1\in F^+(u_4)$ and $col(v_1u_4)\neq 2$.}\cr
	u_1u_2u_3u_{m-1}Pu_4v_3v_2v_1u_m, &\text{if $v_3\in F^+(u_4)$ and $col(v_3u_4)=2$.}\cr
	u_1u_2u_3u_{m-1}Pu_4v_3v_1v_2u_m, &\text{if $v_3\in F^+(u_4)$ and $col(v_3u_4)\neq 2$.}\cr
	u_1v_2u_4Pu_{m-2}v_1u_2u_3u_{m-1}v_3u_m, &\text{if $v_2\in F^+(u_4)$ and $col(v_2u_4)=2$.}\cr
	u_1v_3v_1u_2u_3u_{m-1}Pu_4v_2u_m, &\text{if $v_2\in F^+(u_4)$ and $col(v_2u_4)\neq 2$.}\cr
\end{cases}
\]	
In all cases, $Q$ is a PC Hamilton path of $G$ with $ch(Q)=ch(P)$.
\end{proof}

\noindent
{\bf Proof of Theorem \ref{Thm:Fengetal}.}
It suffices to prove for the case $k=1$. Assume $ch(P)=[u,x,y,u']$. If there
exists $v\in V(C_1)$ such that $col(vu)\neq x$ or $col(vu')=y$, then it
is easy to see the existence of a PC Hamilton path. The remaining
case is that $col(u,C_1)=\{x\}$ and $col(u',C_1)=\{y\}$. By Theorem \ref{Thm:key-1},
we also obtain a PC Hamilton path. The proof is complete. $\hfill\blacksquare$

\section{Proofs of Theorems \ref{Y:97} and \ref{Thm:key-2}}

\noindent
{\bf Proof of Theorem \ref{Thm:key-2}.}
If there are two parallel edges admitting distinct colors
in $G$, then a PC cycle of length 2 is obtained. Now assume
that $G$ is simple. We will prove Theorem \ref{Thm:key-2} by 
the way of contradiction.
Suppose that $G$ contains no PC cycle. We construct an auxiliary
digraph $D$ with $V(D)=V(G)$ and $A(D)=\{uv:$ {\it there are
two internally vertex-disjoint PC paths from $u$ to $v$ with
the distinct starting colors}$\}$. Then $\delta^+(D)\geq 1$.
So $D$ contains a directed cycle $v_1v_2\cdots v_kv_1$ with
$k\geq 2$. Let $P_i$ and $Q_i$ be the two internally vertex-disjoint
PC paths from $v_i$ to $v_{i+1}$ for $i=1,2,\ldots,k$ (indices
are taken modulo $k$). Since $P_i$ and $Q_i$ do not form a PC
cycle and the starting colors are distinct, the two paths must
have a same ending color, say $c_i$. To be specific, we further assume that
the starting color of $P_i$ is always distinct to $c_{i-1}$. Let
$C_i$ be the cycle formed by $P_i\cup Q_i$. Let $v_iP_iv_{i+1}Q_iv_i$
be the positive direction for each cycle $C_i$. For $u,w \in V(C_i)$,
we use $uC_i^+w$ and $uC_i^-w$ to denote the segments on $C_i$ from
$u$ to $w$ along the positive direction and the inverse direction,
respectively.

Since $P_1P_2\cdots P_k$ is a closed walk with consecutive
edges in distinct colors but not a PC cycle, there must exist
a pair of indices $(i,j)$ such that
$(V(C_{i})\setminus \{v_{i+1}\})\cap V(P_{j})\neq \emptyset$
for some $i,j\in [1,k]$ with $j\neq i$ and $|\bigcup_{t=i}^{j} V(P_t)|$ is
smallest, where the indices are taken modulo $k$.

Let $w$ be a vertex in $(V(C_{i})\setminus \{v_{i+1}\})\cap V(P_{j})$
closest to $v_j$ on $P_j$. We can see that either
$v_{i+1}P_{i+1}v_{i+2}\cdots v_jP_jwC^+_iv_{i+1}$
or $v_{i+1}P_{i+1}v_{i+2}\cdots v_jP_jwC^-_iv_{i+1}$ is a PC cycle,
a contradiction. This completes the proof of Theorem \ref{Thm:key-2}.

\begin{definition}[A closure for the existence of PC cycles]\label{def:1}
	Let $G$ be an edge-colored graph with no PC cycle. If joining
	two nonadjacent vertices $u$ and $v$ in $G$ by an edge $e$ of
	a certain color will not produce any PC cycle in $G+e$, then
	replace $G$ by $G+e$. Repeat this process until no more edges
	can be added. Then we say the resulting edge-colored graph is a {\it closure} of $G$.	
\end{definition}

\begin{proof}[\bf Proof of Theorem \ref{Y:97}]
Without loss of generality, assume that $G$ contains no loops
or multiple edges. Let $H$ be a closure of $G$ and let
$N^i_v=\{u\in N_H(v): col_H(uv)=i\}$.

If two vertices $u$ and $v$ are not adjacent in  $H$, then
by Definition \ref{def:1}, no matter what color we assign to
$uv$, the obtained edge-colored graph $H+uv$ would contain a
PC cycle. Thus there are two PC paths $P$ and $Q$ from $u$ to
$v$ with distinct starting colors. Let $u'$ be the vertex in
$V(P)\cap V(Q)$ closest to $u$ on $P$. Then $uPu'$ and $uQu'$
are two internally vertex-disjoint PC paths from $u$ to $u'$
with distinct starting colors in $H$.

Since $H$ contains no PC cycle, by Theorem \ref{Thm:key-2} and the above
analysis, there is a vertex $z\in V(H)$ such that $z$ is adjacent
to all the other vertices in $H$, and for each vertex
$w\in V(H)\setminus\{z\}$ and each PC path $P$ from $z$ to $w$,
the starting color must be the same, namely $col_H(zw)$. This
implies that there is no edge between $N^i_z$ and $N^j_z$ for
any distinct colors $i$ and $j$. Thus each component of $H-z$
is joint to $z$ with only one color. Recall that $G$ is a subgraph
of $H$. So each component of $G-z$ is joint to $z$ with at most one color.
The proof is complete.
\end{proof}
\section{Concluding remarks}
As an analogue to Theorem \ref{Thm:Fengetal}, Cheng, Kano and Wang \cite{CKW:20}  conjectured the following.
\begin{conjecture}[Cheng et al. \cite{CKW:20}]
Each edge-colored complete balanced bipartite graph $K_{n,n}$ contains a PC Hamilton path if it contains a PC 1-path-cycle factor. 
\end{conjecture}
Note that Lemma \ref{lem:AB} still holds when the host edge-colored graph is complete bipartite. However, Lemma \ref{lem:edge} is not applicable to edge-colored $K_{n,n}$ in the present form.

Let $C_1$ and $C_2$ be vertex-disjoint PC cycles in an edge-colored $K_n$.
Looking at Theorem \ref{Thm:Fengetal}, a natural question would be whether
$C_1$ and $C_2$ can be emerged into one PC cycle. Assume that there exists
a coloring $f: V(C_1)\rightarrow col(C_1)$ such that
$col(uv)=f(u)$ or $f(v)$ for all distinct vertices $u,v\in V(C_1)$ and $col(uw)=f(u)$
for all $u\in V(C_1)$ and $w\in V(C_2)$. It is easy to check that there is no
PC cycle $C$ with $V(C)=V(C_1)\cup V(C_2)$.
Such a vertex set $V(C_1)$ was called ``degenerate set'' by the first author when
studying the relationship between edge-colored graphs and digraph
(see \cite{LLZ:20,LiR:21}). B\'{a}nkfalvi and B\'{a}nkfalvi \cite{Bankfalvi:1968}
gave a necessary and sufficient
condition for the existence of a PC Hamilton cycle in a $2$-colored $K_n$. Their
proof implies that the above construction is the only barrier for emerging $C_1$
and $C_2$ in a $2$-colored $K_n$. For cases with at least $3$ colors, the following
small example containing no PC cycle on $V(C_1)\cup V(C_2)$ implies the existence
of other kinds of barriers.
\begin{construction}
$C_1=x_1x_2x_3x_1,C_2=y_1y_2y_3y_1$; $col(x_ix_j)=i, col(y_iy_j)=i
\text{~and~} col(x_iy_j)=i  \text{~for all~} i,j\in \{1,2,3\}$;
$\text{recolor the edge $x_2y_3$ by color $1$}.$
\end{construction}

An interesting problem might be stated as following.
\begin{problem}\label{prob:C1C2}
Let $C_1$ and $C_2$ be vertex-disjoint PC cycles in a $k$-colored
$K_n$($k\geq 3$). Characterize the coloring structure when there is no
PC cycle $C$ satisfying $V(C)=V(C_1)\cup V(C_2)$.
\end{problem}
The characterization given by B\'{a}nkfalvi and B\'{a}nkfalvi
\cite{Bankfalvi:1968} implies a polynomial
algorithm to check whether a $2$-colored $K_n$ contains a PC Hamilton
cycle. In general, Gutin and Kim \cite{Gutin-Kim:2009} made the following conjecture. Solving Problem \ref{prob:C1C2} would be a big step to studying Conjecture \ref{con:Gutin-Kim}.
\begin{conjecture}[Gutin and Kim \cite{Gutin-Kim:2009}]\label{con:Gutin-Kim}
It is polynomial time checkable whether an edge-colored complete graph contains a PC Hamilton cycle.
\end{conjecture}

Recently, Galeana-S\'{a}nchez, Rojas-Monroy, S\'{a}nchez-L\'{o}pez and Imelda Villarreal-Vald\'{e}s \cite{GRSI:22} generalized Yeo's Theorem to $H$-Cycles. Let $G$ be an edge-colored graph and $H$ be a loopless graph with $V(H)=col(G)$. A cycle $C$ in $G$ is called an $H$-cycle if for each pair of consecutive edges $e$ and $f$ in $C$, the colors $col(e)$ and $col(f)$ are adjacent in $H$. Particularly, a PC cycle in $G$ is an $H$-cycle with $H$ being a complete graph on $V(H)=col(G)$. By slightly modifying Theorem \ref{Thm:key-2} and the definition of closure, one can obtain a one-page proof of the main theorem in \cite{GRSI:22}.


\begin{thebibliography}{10}

\bibitem{BG:97}
{J. Bang-Jensen and G. Gutin},
\newblock {Alternating cycles and paths in edge-coloured multigraphs: a survey},
\newblock {\em Discrete Math.,} {\bf 165/166} (1997) 39--60.

\bibitem{BGH:96}
{J. Bang-Jensen, G. Gutin, and J. Huang},
A sufficient condition for a semicomplete multipartite digraph to be Hamiltonian,
\newblock {\em Discrete Math.}, {\bf 161} (1996) 1--12.

\bibitem{BGY:98}
{J. Bang-Jensen, G. Gutin, and A. Yeo},
Properly coloured Hamiltonian paths in edge-coloured complete graphs,
\newblock {\em  Discrete Appl. Math.}, {\bf 82} (1998) 247--250.

\bibitem{Bankfalvi:1968}
{M. B\'{a}nkfalvi and Zs. B\'{a}nkfalvi},
Alternating hamiltonian circuit in two-coloured complete graphs. In
{\em Proc. Colloq. Tihany 1968},
Academic Press (1968) 11--18.

\bibitem{BM:08}
{J.A. Bondy and U.S.R. Murty},
\newblock {\em Graph Theory},
\newblock {Springer Graduate Texts in Mathematics, vol. 244 (2008)}.

\bibitem{BE:76}
{B. Bollob\'as and P. Erd\H{o}s},
\newblock{Alternating Hamiltonian cycles},
\newblock{\em Israel J. Math.}, {\bf 23} (1976), 126--131.

\bibitem{BFMMMPS:19}
{V. Borozan, W. Fernandez de La Vega, Y. Manoussakis, 
C. Martinhon, R. Muthu, H.P. Pham, and R. Saad},
\newblock {Maximum colored trees in edge-colored graphs},
\newblock {\em European J. Combin.,} {\bf 80} (2019) 296--310.

\bibitem{COY:18}
{R. \v{C}ada, K. Ozeki, and K. Yoshimoto},
{A complete bipartite graph without properly colored cycles of length four},
\newblock {\em J. Graph Theory,} {\bf 93} (2020) 168--180.


\bibitem{CKW:20}
{Y. Y. Cheng, M. Kano, and G. Wang},
Properly colored spanning trees in edge-colored graphs,
\newblock {\em  Discrete Math.}, {\bf 343} (2020) Paper No. 111629.

\bibitem{D:78}
{D. E. Daykin},
{Graphs with cycles having adjacent lines of different colors},
\newblock {\em J. Combin. Theory Ser. B,} {\bf 20} (1976), 149--152.

\bibitem{FGGGJR:06}
{J. Feng, H. Giesen, Y. Guo, G. Gutin, T. Jensen, and A. Rafiey},
\newblock Characterization of edge-colored complete graphs with properly colored Hamiltonian paths,
\newblock {\em J. Graph Theory,} {\bf 53} (2006) 333--346.

\bibitem{FLW:19}
{S. Fujita, R. Li, and G. Wang,}
Decomposing edge-colored graphs under color degree constraints,
{\em Combin. Probab. Comput.,}  {\bf 28} (2019) 755--767.

\bibitem{FLZ:18}
{S. Fujita, R. Li, and S. Zhang},
\newblock Color degree and monochromatic degree conditions for short properly colored cycles in edge-colored graphs,
\newblock {\em J. Graph Theory, } {\bf 87} (2018) 362--373.

\bibitem{GRSI:22}
{H. Galeana-S\'{a}nchez, R. Rojas-Monroy, R. S\'{a}nchez-L\'{o}pez and J. Imelda Villarreal-Vald\'{e}s},
H-cycles in H-colored multigraphs,
\newblock{\em Graphs Combin.,} {\bf 38} (2022), Paper No. 62, 20 pp.

\bibitem{GH:83}
{J.W. Grossman and R. H\"{a}ggkvist},
\newblock {Alternating cycles in edge-partitioned graphs},
\newblock {\em J. Combin. Theory Ser. B}, {\bf 34} (1983) 77--81.

\bibitem{Gutin:93}
{G. Gutin},
Finding a longest path in a complete multipartite digraph,
\newblock {\em SIAM J. Discrete Math.}, {\bf 6} (1993) 270--273.

\bibitem{GJSWY:17}
{G. Gutin, M. Jones, B. Sheng, M. Wahlstr\"{o}m, and A. Yeo},
\newblock {Acyclicity in edge-colored graphs},
\newblock {\em Discrete Math.,} {\bf 340} (2017) 1--8.

\bibitem{Gutin-Kim:2009}
G. Gutin and E. J. Kim,
Properly coloured cycles and paths: Results and open problems,
{\em Graph Theory Computational Intelligence and Thought},
Essays Dedicated to Martin Charles Golumbic on the Occasion of His 60th Birthday, (2009) 200--208.

\bibitem{LBYY:21}
{R. Li, H. J. Broersma, M. Yokota, and K. Yoshimoto},
\newblock {Edge-colored complete graphs without properly colored even cycles: a full characterization},
\newblock{\em J. Graph Theory}, {\bf 98} (2021) 110--124.

\bibitem{LBZ:20}
{R. Li, H. J. Broersma, and S. Zhang},
\newblock {Vertex-disjoint properly edge-colored cycles in edge-colored complete graphs},
\newblock{\em J. Graph Theory}, {\bf 94} (2020) 476--493.

\bibitem{LLZ:20}
{R. Li, B. Li, and S. Zhang},
\newblock {A classification of edge-colored graphs based on properly colored walks},
\newblock{\em Discrete Appl. Math.} {\bf 283} (2020) 590--595.

\bibitem{LiR:21}
{R. Li},
\newblock {Properly colored cycles in edge-colored complete graphs
without monochromatic triangle: a vertex-pancyclic analogous result},
\newblock{\em Discrete Math.} {\bf 344} (2021) Paper No. 112573.

\bibitem{L:16}
{A. Lo},
\newblock Properly coloured Hamiltonian cycles in edge-coloured complete graphs,
\newblock{\em Combinatorica,} {\bf 36} (2016) 471--492.

\bibitem{ORR:19}
{G.R. Omidi, G. Raeisi, and Z. Rahimi,}
\newblock {Stars versus stripes Ramsey numbers},
\newblock {\em European J. Combin.,} {\bf 67} (2018) 268--274.

\bibitem{V07}
{L. Volkmann},
\newblock Multipartite tournaments: a survey,
\newblock {\em Discrete Math.,} {\bf 306} (2007) 3097--3129.

\bibitem{WYZ:16}
{Y. Wu, D. Ye, and C-Q Zhang},
\newblock {Uniquely forced perfect matching and unique 3-edge-coloring},
\newblock {\em Discrete Appl. Math.,} {\bf 215} (2016) 203--207.

\bibitem{Y:97}
{A. Yeo},
\newblock A note on alternating cycles in edge-colored graphs,
\newblock{\em J. Combin. Theory Ser. B}, {\bf 69} (1997) 222--225.	
	\end{thebibliography}
\end{document}